\documentclass[11pt]{amsart}
\usepackage[a4paper,margin=3cm]{geometry}
\usepackage[pdftex,pdfborder={0 0 0},colorlinks]{hyperref}
\usepackage{latexsym,lmodern,graphicx,enumitem,amssymb}
\usepackage{epsfig,psfrag}

\newtheorem{thm}[equation]{Theorem}
\newtheorem{lem}[equation]{Lemma}
\newtheorem{prop}[equation]{Proposition}
\newtheorem{cor}[equation]{Corollary}

\newtheorem{rem}[equation]{Remark}
\numberwithin{equation}{section}

\newcommand{\R}{\mathbb{R}}

\renewcommand{\P}{\mathbb{P}}
\newcommand{\E}{\mathbb{E}}
\newcommand{\ep}{\varepsilon}
\newcommand{\Var}{\mathrm{Var}}

\begin{document}

\title[Transport proofs of weighted Poincar\'e  inequalities]
{Transport proofs of weighted  Poincar\'e inequalities for log-concave distributions}
\author{Dario Cordero-Erausquin and Nathael Gozlan}

\date{\today}

\address{
DCE: Institut de Mathématiques de Jussieu and Institut Universitaire de France, 
Universit\'e Pierre et Marie Curie -- Paris 6,  4, place Jussieu,  75252 Paris Cedex 05,France }
\email{dario.cordero@imj-prg.fr}
\address{
NG: Universit\'e Paris Est Marne la Vall\'ee - Laboratoire d'Analyse et de Math\'e\-matiques Appliqu\'ees (UMR CNRS 8050), 5 bd Descartes, 77454 Marne la Vall\'ee Cedex 2, France}
\email{nathael.gozlan@univ-mlv.fr}

\thanks{This work was partially supported by the grants ANR 2011 BS01 007 01 (GeMeCOD) and ANR 10 LABX 58}

\maketitle

 \begin{abstract}We prove, using optimal transport tools, weighted Poincar\'e inequalities for log-concave random vectors satisfying some centering conditions. We recover by this way similar results by Klartag and Barthe-Cordero-Erausquin for log-concave random vectors with symmetries. In addition, we prove that the variance conjecture is true for increments of log-concave martingales.
% that any $n$ dimensional log-concave vector $X$ such that $\E[X_i| X_1,\ldots,X_{i-1}]=0$ for all $i\leq n$ satisfies the inequality
% \[
%\E[f^2(X)] -\E[f(X)]^2 \leq a \sum_{i=1}^n \E\left[ \E[X_i^2|X_1,\ldots,X_{i-1}] \partial_i f(X)^2\right],
% \]
% for all smooth function $f$ such that $\E[X_if(X)|X_1,\ldots,X_{i-1}]=0$ for all $i\in \{0,\ldots,n\},$
%where $a>0$ is some universal constant. This extends some previous result by Klartag about $1$-conditional log-concave random vectors.
\end{abstract}

\section{Introduction}
In all the paper, if $X=(X_1, \ldots,X_n)$ is a random vector defined on some probability space $(\Omega,\mathcal{A},\P)$ with values in $\R^n$ and $h:\R^n\to\R$ is an Borel (bounded or nonnegative) function, we use the following notation for the conditional expectations: 
\[
\E_{i}[h(X)]:= \E[h(X)|X_1,\ldots,X_{i-1}],
\] 
with the convention that $\E_0[h(X)]=\E[h(X)].$ To any random vector $X$, we associate the random vector $\overline{X}$ defined as follows:
\[
\overline{X}_i = X_i-\E_{i-1}[X_i],\qquad \forall i \in \{1,\ldots,n\}.
\]
This recentering procedure will play an important role in all the paper (see also \cite{BGRS13} for another application). We aim at proving Poincar\'e and transport inequalities for $\overline X$, when $X$  is log-concave. 

\smallskip
Recall that a random vector $X$ with values in $\R^n$ is log-concave if for all non-empty compact sets $A,B \subset \R^n$, it holds
\[
\P(X \in (1-t)A + tB) \geq \P(X\in A)^{1-t}\,\P(X\in B)^t,\qquad \forall t\in [0,1].
\]
According to a celebrated result of Borell \cite{Bor74,Bor75}, a random vector $X$ is log-concave if and only if there is an affine map $\ell : \R^k \to \R^n$, $k\leq n$ and a random vector $Y$ taking values in $\R^k$ such that $X=\ell (Y)$ and $Y$ has a density of the form $e^{-V}$ with respect to the Lebesgue measure on $\R^k$, where $V:\R^k\to \R\cup\{+\infty\}$ is a convex function. In what follows, by an ``$n$-dimensional log-concave random vector'', we will understand a vector $X$ satisfying the conditions above with $k=n$ (and $\ell =\mathrm{Id}$).
\smallskip

The main result of this note is that the class of all random vectors $\overline{X}$ with $X$ $n$-dimensional and log-concave satisfies a general weighted Poincar\'e inequality.
\begin{thm}\label{main result}
There exists a numerical constant $a>0$ such that for any $n$-dimensional log-concave random vector $X$, it holds
\begin{equation}\label{eq:wP}
\Var(f(\overline{X}))\leq a \sum_{i=1}^n \E\left[ \E[\overline{X}_i^2|\overline{X}_1,\ldots,\overline{X}_{i-1}] \partial_i f(\overline{X})^2\right] 
% a\sum_{i=1}^n \E\left[\E_{i-1}[X_if(X)] \partial_if(X) \right]
\end{equation}
for all locally-Lipschitz $f:\R^n\to \R$ belonging to $\mathbb{L}_2(\overline X)$, where $\Var(Y):= \E[Y^2]-\E[Y]^2$ denotes the variance of a real valued random variable $Y.$
In particular, if $X$ is such that $\E_{i-1}[X_i]=0$ for all $i\in\{1,\ldots,n\}$, then $\overline{X}=X$ and it holds
\begin{equation}\label{eq:wP2}
\Var(f(X))\leq a \sum_{i=1}^n \E\left[ \E_{i-1}[X_i^2] \partial_i f(X)^2\right] 
% a\sum_{i=1}^n \E\left[\E_{i-1}[X_if(X)] \partial_if(X) \right]
\end{equation}
\end{thm}
\begin{rem}
If the operation $X \mapsto \overline{X}$ was preserving log-concavity then, of course, \eqref{eq:wP} would follow from \eqref{eq:wP2} applied to $\overline{X}.$ It is not difficult to find examples of log-concave random vectors $X$ such that $\overline{X}$ is not log-concave anymore. A random vector $X$ such that $\overline{X}=X$ can be interpreted as a sequence of martingale increments (see Section 3 for more details).
\end{rem}

Theorem \ref{main result} is reminiscent of recent results of Klartag \cite{Kla13} and of Barthe and Cordero-Erausquin \cite{BCE13} which were based on $\mathbb{L}_2$ methods.  The objective of this note is to give alternative proofs of variants of some of the results from \cite{Kla13,BCE13} using mass transport arguments. 
\smallskip

Recall that a random vector $X$ is \emph{unconditional} when $X=(X_1,\ldots, X_n)$ has the same law as $(\ep_1X_1,\ldots,\ep_n X_n)$ for any choice of $\ep_i=\pm 1$. Since unconditional random vectors satisfy $\E_{i-1}[X_i]=0$ for all $1\leq i\leq n$, Theorem \ref{main result} can be seen as an extension of the following result by Klartag \cite{Kla13}: for any log-concave and unconditional random vector $X$, it holds
\begin{equation}\label{eq:KlartagWP}
\Var(f(X))\leq c \sum_{i=1}^n \E\left[ (X_i^2 +\E[X_i^2])\partial_i f(X)^2\right],
\end{equation}
for all $f:\R^n\to\R$ smooth enough, where $c>0$ is some absolute constant. 
Moreover, when $f$ is itself unconditional (i.e $f(\ep_1x_1,\ldots,\ep_n x_n) = f(x_1,\ldots,x_n)$ for all $\ep_i=\pm 1$), then the terms $\E[X_i^2]$ can be removed from the right hand side of \eqref{eq:KlartagWP}. Note that in \cite{Kla13}, Klartag also obtains weighted Poincar\'e inequalities for a larger class of unconditional distributions with a density of the form $e^{-\phi}$ with $\phi:\R^n\to\R$ whose restriction to $\R_+^n$ is $p$ convex (i.e $x\mapsto \phi(x_1^{1/p},\ldots,x_n^{1/p})$ is convex). Inequalities of the form \eqref{eq:KlartagWP} were also  investigated in details in the recent paper \cite{BCE13}. There, the authors  establish general weighted Poincar\'e inequalities for classes of probability measures invariant by a subgroup of isometries, not only the coordinate reflections. \smallskip

Note that~\eqref{eq:wP2}  applies to random vectors having less symmetries than unconditional random vectors. For instance, if the $X_i$ are independent mean zero and variance one log-concave random variables then $\E_{i-1}[X_i]=0$ for all $i$, whereas $X$ does not have any particular symmetry. In this case,  the conclusion \eqref{eq:wP2} of Theorem \ref{main result} is consistent with the Poincar\'e inequality obtained using the (elementary) tensorisation property of the Poincar\'e inequality.
%Note also that  if one assumes that $X$ is unconditional and in addition that its density is of the form $e^{-V}$ for some $V$ of class $\mathcal{C}^2$ on $\R^n$ and such that $\mathrm{Hess}\,V \geq \rho \mathrm{Id}$, for some $\rho>0$, then it is not difficult to see that $\E_{i-1}[X_i^2] \leq 1/\rho$ for all $i\in \{1,\ldots,n\}$;  Therefore the conclusion \eqref{eq:wP2} of Theorem \ref{main result} is compatible with the well known fact that random vectors $X$ with a density as above (not necessarily unconditional) enjoy the Poincar\'e inequality 
%\[
%\Var(f(X))\leq \frac{1}{\rho} \E\left[|\nabla f|^2(X)\right],
%\]
%for all smooth function $f.$
%\smallskip

%Our Theorem \ref{main result} is also very close from Theorem 4 of \cite{BCE13}. 
%This latter result from \cite{BCE13} easily implies, for instance, that if $X$ has a density of the form $e^{-V}$ with $\mathrm{Hess}\,V\geq \rho \mathrm{Id}$, for some $\rho\geq0$ then it holds
%\begin{equation}\label{BCE}
%\Var(f(X)) \leq a\sum_{i=1}^n \E\left[ \frac{C_i(X)}{1+\rho C_i(X)} \partial_if(X)^2\right],
%\end{equation}
%where $C_{i}(X) = \E[X_i^2 | (X_j)_{j\neq i}] - \E[X_i| (X_j)_{j\neq i}]^2$, for all function $f:\R^n\to \R$ enjoying the condition
%\[
%\E[\partial_if(X) | (X_j)_{j\neq i}]=0,\qquad \forall i \in \{1,\ldots,n\}.
%\]
%So when $\overline{X}=X$ Theorem \ref{main result} gives a weighted Poincar\'e inequality similar to \eqref{BCE} but which holds true for all $f:\R^n \to \R$.

Theorem \ref{main result} also easily implies some variance estimates for log-concave random vectors.
\begin{cor}\label{cor:variance}
There exists a universal constant $b>0$ such that if $X$ is an $n$-dimensional log-concave random vector, then, denoting by $|\,\cdot\,|$ the standard Euclidean norm on $\R^n$, it holds
\begin{equation}\label{variance estimate0}
\Var\left(|\overline{X}|^2\right) \leq b \sum_{i=1}^n \E[\overline{X}_i^4]\leq 16b \sum_{i=1}^n \E[X_i^4].
\end{equation}
%In particular,
%\begin{equation}\label{variance estimate}
%\Var\left(|\overline{X}|^2\right) \leq bn.
%\end{equation}
In particular, when  $\E[X_i^2]=1$ for all $i\in \{1,\ldots,n\}$, we have
\begin{equation}\label{variance estimate01}
\Var\left(|\overline{X}|^2\right) \le c n
\end{equation}
and  if in addition $X$ satisfies $\E_{i-1}\left[X_i\right]=0$ for all $i$, then
\begin{equation}\label{variance estimate2}
\Var\left(|X|^2\right) \leq c n,
\end{equation}
for some other universal constant $c.$
\end{cor}
%The second part of this result extends a similar result by Klartag in the unconditional case \cite{Kla09,Kla13}. 
%Under the slightly different condition $\E[X_i | (X_j)_{j\neq i}]=0$ for all $i$, the weighted Poincar\'e inequality \eqref{BCE} of Barthe and Cordero-Erausquin also gives that $\Var(|X|^2)\leq bn$ for some universal constant $b>0.$ We refer to \cite{BCE13} for other classes of examples of log-concave probability measures enjoying the variance conjecture.

\smallskip
%A more symmetric set of conditions implying \eqref{centering} is
%\[
%\E[X_i^p| X_1,\ldots,X_{i-1},X_{i+1},\ldots,X_n] = 0,\text{ for } p \in \{1;3\}, \quad \text{and}\quad \E[X_i^2]=1,
%\]
%for all $i\in \{1,\ldots,n\}$.

The inequality \eqref{variance estimate2} on the variance immediately yields to the following concentration in a thin shell estimate
\[
\P\left( \left||X| - \sqrt{n} \right| \geq t \sqrt n \right)\leq  b\,  e^{-c\,  n^{1/4} \sqrt t},\qquad \forall t>0.
\]
This type of concentration inequalities plays a central role in the proof of the central limit theorem for log-concave random vectors \cite{ABP03,Kla07,Kla10,Bob03}.
\smallskip

Corollary \ref{cor:variance} is also motivated by the so called \emph{variance conjecture}. 
Recall that a random vector $X$ is said \emph{isotropic} if $\E[X]=0$ and $\E[X_iX_j]=\delta_{i,j}$ for all $i,j \in \{1,\ldots,n\}.$ The variance conjecture asserts that any log-concave and isotropic random vector $X$ satisfies \eqref{variance estimate2} for some universal positive  constant $b$. This conjecture was shown to be true in restriction to the class of unconditional log-concave random vectors by Klartag \cite{Kla09,Kla13}. We refer to \cite{BCE13} and \cite{A-GB13} for other subclasses of log-concave distributions satisfying the variance conjecture. The best (dimensional) estimate in date is due to Gu\'edon and Milman \cite{GM11} who proved that $\Var(|X|)\leq bn^{2/3}$ for any isotropic log-concave random vector $X$. The variance conjecture is a weak form of a celebrated conjecture by Kannan, Lovasz and Simonovits \cite{KLS95} stating that 
any log-concave and isotropic random vector $X$ satisfies a Poincar\'e inequality
\[
 \Var(f(X)) \leq a\E\left[ |\nabla f|^2(X)\right],\qquad \forall f \text{ smooth enough},
\]
for some universal constant $a>0.$ According to a remarkable recent result of Eldan \cite{Eld13}, the variance conjecture implies the KLS conjecture up to some $\log(n)$ factor.
\smallskip

Corollary \ref{cor:variance} thus shows that the variance conjecture is satisfied on the class of isotropic log-concave random vectors such that $\overline{X}=X$ (see also \cite[Theorem 4]{BCE13} and Remark \ref{rem BCE} below for a related result). It is not difficult to see that this class is strictly larger than the class of unconditional isotropic and log-concave random vectors (some informations on log-concave random vectors such that $\overline{X}=X$ can be found in Proposition \ref{propsym} and Remark \ref{remsym} below).
%Moreover, it follows from \eqref{variance estimate0} that for any log-concave random vector $X$ such that $\E[X_i^2]=1$ for all $i$, 
%the variance of $|\overline{X}|^2$ is always bounded from above by a universal constant times the dimension.
For general log-concave random vectors $X$, let us mention that it is always at least possible to bound $\Var(|X|^2)$ in terms of $\Var(|\overline{X}|^2)$ and of $\Var(|X'|^2)$, where the ``reduced'' random vector $X'$ is defined by
\[X'_i = \E_{i-1}[X_i],\qquad \forall i \in \{1,\ldots,n\}.\]
The basic observation behind the following elementary result is that $X=\overline{X}+X'$ is an orthogonal decomposition of $X$ in the space $H:=\mathbb{L}_2(\Omega,\mathcal{A},\P ; \R^n)$ of square integrable $n$-dimensional random vectors. More precisely, for any $X \in H$, the vector $\overline{X}$ is the orthogonal projection of $X$ onto the linear subspace $H_0(X)=\{Y \in H ; \E[Y_i|X_1,\ldots,X_{i-1}]=0,\  \forall i\in \{1,\ldots,n\}\}$ (the space of random sequences that are matingale increments with respect to the filtration $\sigma(X_1,\ldots,X_i)$, $0\leq i \leq n-1$.)  We will prove the following useful identity
\begin{multline}
\Var(|X|^2) = \Var(|\overline{X}|^2) + \Var(|X'|^2) + 2 \mathrm{Cov}(|\overline{X}|^2,|X'|^2) \\
+ 4 \E[(\overline{X}\cdot X')^2] + 4\E[|\overline{X}|^2\overline{X}\cdot X'] + 4\E[|X'|^2\overline{X}\cdot X'],
\label{identity}
\end{multline}
from which one deduces the following result:
%\begin{lem}\label{orthogonal}
%For any random vector $X$ defined on some probability space $(\Omega,\mathcal{A},\P)$, the decomposition $X=\overline{X} + X'$ is orthogonal in $\mathbb{L}_2(\Omega,\mathcal{A},\P ; \R^n)$ of square integrable $n$-dimensional random vectors. 
%In particular, it holds
%\[ 
%\Var\left(|X|^2\right) = \Var\left(|\overline{X}|^2\right) +2\mathrm{Cov}\left(|\overline{X}|^2,|X'|^2\right) + \Var\left(|X'|^2\right).
%\]
%Therefore,
%\begin{equation}\label{eq:encadrement}
%\left(\sqrt{\Var\left(|\overline{X}|^2\right)} - \sqrt{\Var\left(|X'\vphantom{\overline{X}}|^2\right)}\right)^2\leq \Var\left(|X|^2\right)\leq 2 \Var\left(|\overline{X}|^2\right) + 2\Var\left(|X'|^2\right).
%\end{equation}
%\end{lem}
%
%In particular, \eqref{eq:encadrement} and Corollary \ref{cor:variance} immediately imply the following
\begin{cor}\label{encadrement}
If $X$ is an isotropic and log-concave random vector in $\R^n$, and $X'$ is defined as above, then 
\[
\Var\left(|X|^2\right) \leq a\left(n + \Var\left(|X'|^2\right) \right)\qquad\text{and}\qquad \Var\left(|X'|^2\right)\leq a\left(n+\Var\left(|X|^2\right)\right),
\]
for some universal constant $a.$
\end{cor}
It follows that the variance conjecture is (technically) equivalent to the existence of a universal constant $b>0$ such that for any isotropic log-concave random vector $X$, 
\[
\Var\left(|X'|^2\right) \leq bn.
\] 
It would be of some interest to see if for some specific classes of vectors $X$, the variance term $\Var(|X'|^2)$ can be estimated by some power of $n$.

\smallskip

The proof of Theorem \ref{main result} is based on mass transport. More precisely, we will establish a transport-entropy inequality (Theorem~\ref{transport inequality}) which is of independent interest, of the form
%\begin{equation}\label{transport-intro}
$$ \mathcal{T}_\mu(\overline{\mu},\overline{\nu}) \leq D(\nu\,\|\,\mu),\qquad \forall \nu,$$
%\end{equation}
where $\overline{\mu}$ and $\overline{\nu}$ are the laws of random vectors $\overline{X}$ and $\overline{Y}$, with $X,Y$ distributed according to $\mu$ and $\nu$. The optimal transport cost $\mathcal{T}_\mu$ will be of the form
\[
\mathcal{T}_\mu(\nu_0,\nu_1) = \inf _{\pi \in C(\nu_0,\nu_1)} \iint c_\mu(x,y)\,\pi(dxdy),
\]
for a particular  cost function $c_\mu$  (precise definitions will be given later).
Then,  Theorem~\ref{main result} will follow from this transport-entropy inequality by a standard linearization procedure.  The argument towards our transport inequality will use an  \emph{above tangent lemma}   introduced by Cordero-Erausquin \cite{CE02} which is a handy tool to prove classical functional inequalities (Log-Sobolev, Talagrand) for uniformly log-concave random vectors and to recover the celebrated HWI inequality of Otto and Villani \cite{OV00}.  

Let us mention here a byproduct of this approach in terms of transport inequalities involving the classical $W_2$ distance (definitions are recalled below).
\begin{thm}\label{thm:transport T2}
There exists a universal constant $c$ such that for any $n$ dimensional log-concave random vector $X$ taking values in the hypercube $[-R,R]^n$, $R>0$, it holds
\[
W_2^2(\overline{\mu},\overline{\nu}) \leq c R^2 D(\nu\,\|\, \mu),
\] 
for all probability measures $\nu$ on $\R^n$, where $\overline{\mu}$ and $\overline{\nu}$ denote respectively the laws of $\overline{X}$ and $\overline{Y}$, $Y$ being distributed according to $\nu$.
\end{thm}
Theorem \ref{thm:transport T2} can be considered as a variant of results by Eldan and Klartag \cite[Theorem 6.1]{EK} and by Klartag \cite[Theorem 4.2]{Klartag-cube}. Let us mention that the present paper uses techniques of proof very similar to those involved in \cite{EK,Klartag-cube}.
To be more precise, Theorem 6.1 of \cite{EK} gives a similar inequality when $\mu$ and $\nu$ are both unconditional and log-concave. In their statement, the relative entropy is replaced by $\int_{[-R,R]^n} H(f,g)-1$, where $H(f,g)=\sup_{x\in \R^n}\sqrt{f(x+y)g(x-y)}$, denoting by $f,g$ the densities of $\mu$ and $\nu$ with respect to Lebesgue. This quantity is relevant in their study of the stability of the Brunn-Minkowski inequality. 
In Theorem 4.2 of \cite{Klartag-cube}, Klartag obtains the inequality 
\[
W_2^2(\mu,\nu) \leq c L^2 D(\nu\,\|\, \mu),\qquad \forall \nu
\]
for all log-concave probability measures $\mu$ supported on the hypercube $Q=[-1,1]^n$ and such that in addition the density $f$ of $\mu$ with respect to Lebesgue satisfies for some $L\geq 1$
\[
f((1-t) x +t y) \leq L [(1-t) f(x) +t f(y)],\qquad \forall t \in [0,1],
\]
for all $x,y \in Q$ with $x-y$ proportional to one of the standard basis vectors $e_i$. This condition is for instance realized with $L=e^{M/8}$ if $f=e^{-V}$ for some smooth convex function $V:Q\to \R$ such that $\sup_{i\leq n}\sup_{x\in Q}\partial_i^2 V(x) \leq M$ for some $M\geq 0$. 
\bigskip

The paper is organized as follows. In Section 2, we  gather various observations on the relations between $\overline{X}$ and $X$ for log-concave random vectors. In Section 3, we give some background on the mass transportation tools that are used to establish our general transport-inequality, which is stated and proved in Section 4, together with Theorem~\ref{thm:transport T2}.  Then, in Section 5  we linearize this transport-entropy inequality and establish Theorem \ref{main result}. In the final Section~6, we explain how to derive the Corollaries \ref{cor:variance} and \ref{encadrement} on the variance.

%prove Corollaries \ref{cor:variance} and \ref{encadrement}. Section 3. In Section 4, we establish a general transport-entropy inequality for log-concave distributions. Finally, in Section 5, we linearize this transport-entropy inequality to establish Theorem \ref{main result}.

\section{Some observations about log-concave random vectors such that $\overline{X}=X$}
First, we begin with a straightforward proposition identifying the class of random vectors such that $\overline{X}=X$ as increments of martingales.
%We have then the following straightforward proposition
\begin{prop}
A random vector $X=(X_1,\ldots,X_n)$ is such that $\overline{X}=X$ if and only if $M=(M_0,M_1,\ldots,M_n)$, with $M_0=0$ and $M_k=\sum_{i=1}^k X_i$ is a martingale with respect to the increasing sequence of sub-sigma fields $\mathcal{F}_{k}= \sigma(M_0,\ldots,M_{k})$, $k\in \{0,\ldots,n\}.$ 
\end{prop}
The proof if left to the reader.
\smallskip

If $M:=(M_0,M_1,\ldots,M_n)$ is a martingale, we denote by $\Delta_i = M_i-M_{i-1}$, $i\in \{1,\ldots,n\}$ the increments of $M$. The quadratic variation process of $M$ is then defined by $[M]_k = \sum_{i=0}^{k} \Delta_i^2$, for all $k \in \{0,1,\ldots,n\}.$ With these definitions, Corollary \ref{variance estimate0} can be restated as follows. 
\begin{prop}
There exists a universal constant $c>0$ such that for all martingale $M=(M_0,M_1,\ldots, M_n)$ such that $M_0=0$ and $(M_1,\ldots,M_n)$ has a log-concave density, it holds
\[
\mathrm{Var}([M]_k) \leq c \sum_{i=1}^k \E[\Delta_i^4],\qquad \forall k\leq n.
\]
\end{prop}
\proof
Since the class of log-concave random vectors is stable under affine transformations, it follows that $(M_1,\ldots,M_n)$ has a log-concave density if and only if $(X_1,\ldots,X_n)$ with $X_i=\Delta_i$ has a log-concave density. The result then follows immediately from Corollary \ref{variance estimate0}.
\endproof

We now collect in the following proposition some elementary informations on log-concave random vectors $X$ such that $\overline{X} = X.$
\begin{prop}\label{propsym}\ 
\begin{enumerate}
\item If $X,Y$ are two independent log-concave random vectors (defined on the same probability space) such that $\overline{X}=X$ and $\overline{Y}=Y$, then $\overline{X+Y}=X+Y.$
\item If $X$ is a log-concave random vector with values in $\R^n$ then $\overline{X}=X$ if and only if $\E[X]=0$ and for all $k\in \{1,\ldots,n-1\}$, $\E[X | X_1,\ldots, X_k] = (X_1,\ldots,X_k,0,\ldots,0).$
In particular, if $C\subset \R^n$ is a bounded convex body and $X$ is uniformly distributed over $C$, then $\overline{X} =X$ if and only if the barycenter of $C$ is at $0$ and for all $x=(x_1,\ldots,x_n) \in \R^n$
\[
\mathrm{Bar} (C\cap \{(x_1,\ldots,x_k)\} \times \R^{n-i}) = (x_1,\ldots,x_k,0\ldots,0),\qquad \forall k\in\{1,\ldots,n-1\},
\]
(whenever this section is not empty).
In particular, $C$ is symmetric with respect to the hyperplane $\{x_n=0\}.$

\item If $C\subset \R^2$ is a bounded convex body with barycenter at $0$ and $X$ is uniformly distributed over $C$, then $\overline{X}$ is uniformly distributed over the convex body $\overline{C}$ obtained from $C$ by applying Steiner symmetrization with respect to the axis $D=\R \times \{0\}.$ In particular $\overline{X}=X$ if and only if $C$ is symmetric with respect to $D$.
\end{enumerate}
\end{prop}
\proof
(1) It is well known that $X+Y$ is still log-concave. Let us show that $\overline{X+Y}=X+Y$. Let $i \in \{2,\ldots, n\}$ and take $f:\R^{i-1}\to\R$ a bounded measurable test function, then it holds 
\[
\E\left[X_if(X_1+Y_1,\ldots,X_{i-1}+Y_{i-1})\right] = \E_{X}[X_i \E_{Y}[f(X_1+Y_1,\ldots,X_{i-1}+Y_{i-1})]] = 0.
\]
Similarly, $\E\left[Y_if(X_1+Y_1,\ldots,X_{i-1}+Y_{i-1})\right] =0$. Therefore, 
$\E\left[(X_i+Y_i)f(X_1+Y_1,\ldots,X_{i-1}+Y_{i-1})\right] =0$, and since this holds for all test function $f$, one concludes that $\E_{i-1}[(X+Y)_i]=0$ for all $i$ and so $\overline{X+Y}=X+Y$.\\
(2) The second point follows easily from the fact that for $k\leq i-1$, 
\[
\E[X_i | X_1,\ldots,X_k] = \E\left[\E[X_i | X_1,\ldots,X_{i-1}] | X_1,\ldots, X_{k}\right].
\]
(3) Observe that $C$ can be written as $C=\{ (x_1,x_2)\in \R^2 ; x_1 \in [\alpha,\beta], a(x_1)\leq x_2 \leq b(x_1) \}$, for some $\alpha<\beta$, and where $a:[\alpha,\beta] \to \R$ is a concave function and $b:[\alpha,\beta] \to \R$ is a convex function. Recall that the Steiner symmetrization of $C$ with respect to $D$ is the set $\overline{C}$ defined by 
\[
\overline{C} = \left\{(x_1,x_2)\in \R^2 ; x_1 \in [\alpha,\beta], \frac{1}{2}(a(x_1)-b(x_1))\leq x_2 \leq \frac{1}{2}(b(x_1)-a(x_1))  \right\}.
\]
Since the function $a-b$ is convex, the set $\overline{C}$ is convex. Moreover $\overline{X} = \left(X_1, X_2 - \frac{1}{2}(a(X_1)+b(X_1)) \right)$ and so for all bounded measurable test function $f:\R^2 \to \R$
\begin{align*}
\E\left[ f(\overline{X})\right] &= \frac{1}{\mathrm{Vol}(C)}\int_{\alpha}^\beta \int_{a(x_1)}^{b(x_1)}f(x_1,x_2-  \frac{1}{2}(a(x_1)+b(x_1)))\,dx_2dx_1\\
&= \frac{1}{\mathrm{Vol}(C)}\int_{\alpha}^\beta \int_{\frac{1}{2}(a(x_1)-b(x_1))}^{\frac{1}{2}(b(x_1)-a(x_1))}f(x_1, y_2)\,dy_2dx_1\\
& = \frac{1}{\mathrm{Vol}(C)} \int_{\overline{C}} f(y_1,y_2)\,dy_1dy_2.
\end{align*}
This shows that $\overline{X}$ is uniformly distributed on $\overline{C}.$
\endproof

\begin{rem}\label{remsym}
As we already mentioned, the class of log-concave random vectors such that $\overline{X}=X$ already contains unconditional log-concave random vectors and log-concave random vectors with centered independent components.  Using the properties above, it is possible to give other examples of log-concave random vectors such that $\overline{X}=X$ in arbitrary large dimension.
Namely, observe that if $X$ is a log-concave random vector taking values in $\R^k$ and such that $\overline{X}=X$, then  it is easy to check that for all $i\in \{1,\ldots,k+1\}$, the random vector $X^{i}$ defined by 
\[
X^{i} = (X_1,\ldots,X_{i-1},0,X_i,\ldots,X_k) \in \R^{k+1}
\]
is still log-concave and verifies $\overline{X^{i}}=X^{i}.$ Thanks to point (1) of Proposition \ref{propsym}, one thus sees that if $X_1,\ldots,X_{k+1}$ are independent log-concave random vectors with values in $\R^k$ and such that $\overline{X_i}=X_i$, then the random vector $Y= X_1^{1}+X_2^{2} +\cdots+X_{k+1}^{k+1}$ is still log-concave and satisfies $\overline{Y}=Y.$ Initializing this construction with $k=2$ with 
the help of point (3) of Proposition \ref{propsym} and iterating the process gives rise to a large class of non-trivial examples of log-concave random vectors such that $\overline{X}=X.$\\
\end{rem}

\section{Some background on mass transport}
\label{sect:background transport}
\label{section Knothe}

The key lemma used in \cite{CE02} is the so called \emph{above tangent lemma} recalled below. In what follows, the relative entropy (also called the Kullbak-Leibler distance) of $\nu$ with respect to $\mu$ is defined by 
\begin{equation}\label{entropy}
D(\nu \,\|\,\mu) = \int \log \frac{d\nu}{d\mu}\,d\nu,
\end{equation}
if $\nu$ is absolutely continuous with respect to $\mu$ (otherwise, we set $D(\nu\,\|\,\mu)=\infty$).
\begin{lem}[\cite{CE02}]\label{CE}
If $\mu$ is a probability measure on $\R^n$ absolutely continuous with respect to the Lebesgue measure with a density of the form $\mu(dx) = e^{-V(x)}\,dx$ where $V:\R^n\to\R$ is a function of class $\mathcal{C}^2$ such that $\mathrm{Hess}\,V\geq \rho$, $\rho \in \R$, then for all compactly supported probability measures $\nu_0,\nu_1$ absolutely continuous with respect to $\mu$, it holds 
\begin{multline}\label{eq:CE}
D(\nu_1\,\|\,\mu)
 \geq 
D(\nu_0\,\|\,\mu) + \int \left<\nabla \frac{d\nu_0}{d\mu}(x),Tx-x\right>\,\mu(dx) +
\frac{\rho}{2} \int |Tx - x|^2\,\nu_0(dx)\\ + 
\int \big({\rm Tr}\, (DT_x - {\rm I}_n) - \log |DT_x|\big)\,\nu_0(dx),
\end{multline}
where $T:\R^n\to\R^n$ pushes forward $\nu_0$ onto $\nu_1$ and defines a "suitable" change of variables.
%such that $DT_x$ has a non negative spectrum for all $x\in \R^n.$
\end{lem}

First let us recall the classical applications of \eqref{eq:CE}. In \cite{CE02}, the inequality \eqref{eq:CE} was applied with the Brenier map $T$ (see \cite{Vil09}), that is to say the $\nu_0$ almost surely unique map $T$ achieving the infimum in the definition of the square Kantorovich distance $W_2$: 
\[
\int |Tx-x|^2\,\nu_0(dx) = \inf_{\pi \in C(\nu_0,\nu_1)} \iint |y-x|^2\,\pi(dxdy):= W_2^2(\nu_0,\nu_1),
\]
where $C(\nu_0,\nu_1)$ denotes the set of all couplings of $\nu_0,\nu_1$, (i.e probability measures $\pi$ on $\R^n\times\R^n$ having $\nu_0$ and $\nu_1$ as marginals). A fundamental property of the Brenier map $T$ is that it is the gradient of a convex function: there exists $\phi:\R^n\to\R$ convex such that $T(x)=\nabla\phi(x)$ for $\nu_0$ almost every $x \in \R^n.$ As a consequence of the inequality $\log(\lambda) \leq \lambda-1$, $\lambda>0$ and of the fact that $DT_x =\mathrm{Hess}_x\,\phi$ has a non-negative spectrum, the last term in \eqref{eq:CE} is always non-negative (assuming for simplicity that $T$ is smooth). So \eqref{eq:CE} becomes
\begin{equation}\label{eq:CEbis}
D(\nu_1\,\|\,\mu)
 \geq 
D(\nu_0\,\|\,\mu) + \int \left<\nabla \frac{d\nu_0}{d\mu}(x),Tx-x\right>\,\mu(dx) +
\frac{\rho}{2} W_2^2(\nu_0,\nu_1).
\end{equation}
Inequality \eqref{eq:CEbis}, which expresses in some sense that the graph of the map $D(\,\cdot\,\|\,\mu)$ lies above its tangent, is also related to the notion of displacement-convexity of the relative entropy along $W_2$ geodesics (see \cite{McC97,Vil09}). 
When $\rho>0$, interesting consequences can be drawn from the inequality above. For instance, choosing $\nu_0=\mu$ yields to the following transport-entropy inequality
\[
W_2^2(\nu_1,\mu) \leq \frac{2}{\rho} D(\nu_1\,\|\, \mu),\qquad \forall \nu_1.
\]
This type of inequalities goes back to the works by Marton \cite{Mar86} and Talagrand \cite{Tal96a} (see \cite{Led01,Vil09,GL10} for an introduction to the subject).
On the other hand, choosing $\nu_1=\mu$ it is not difficult to derive from \eqref{eq:CEbis} the logarithmic-Soblev inequality (see \cite{CE02,Bar10, GL10} for details)
\[
D(\nu_0\,||\,\mu) \leq \frac{2}{\rho} \int \frac{|\nabla h_0|^2}{h_0}\,d\mu,\qquad \forall \nu_0=h_0\,\mu.
\]
We refer to \cite{BK08,BGRS13} for other applications and variants of \eqref{eq:CE} and \eqref{eq:CEbis}.
\smallskip

In this paper, we will use \eqref{eq:CE} with $\rho=0$ and $\nu_0=\mu$:
\[
D(\nu_1\,\|\,\mu)
 \geq 
\int \big({\rm Tr}\, (DT_x - {\rm I}_n) - \log |DT_x|\big)\,\mu(dx).
\]
But as a main difference, we will rather use as $T$ the Knothe map \cite{Kno57} between $\mu$ and $\nu_1.$ 

Let us  recall the definition of the Knothe transport between two probability measures.
If $\mu,\nu$ are two Borel probability on $\R$ and $\mu$ has no atom, then there exists a unique non-decreasing and left continuous map $T: \R \to [-\infty,\infty]$ transporting $\mu$ on $\nu$ in the sense that $\int f(T)\,d\mu =\int f\,d\nu$ for all say bounded continuous function $f$. This map $T$ is given by 
\[
T(x) = F_\nu^{-1} \circ F_{\mu}(x),\qquad \forall x\in \R.
\]
where, for $x\in \R$ and $t\in [0,1]$,
\[
F_\mu(x)=\mu(-\infty,x]\qquad \text{and}\qquad F_\mu^{-1} (t) = \inf\{ x \in \R ; F_\mu(x)\geq t\} \in [-\infty,\infty].
\]
The map $T$ takes finite values $\mu$ almost surely. Let us mention that the map $T$ achieves the minimum value in a large class of optimal transportation problems (see for instance \cite{RR98}). This fact will not be explicitly used in the sequel.
\smallskip

The Knothe transport map is a multidimensional extension of this one dimensional transport. To define it properly, we need to introduce the following notation. 
If $\mu$ is a probability measure on $\R^n$ and $X=(X_1,\ldots,X_n)$ is a random vector of law $\mu$, we will denote by $\mu_i$ the law of $(X_1,\ldots,X_i)$.
For $i\geq 2$ and $x_1,\ldots,x_{i-1} \in \R$, we denote by $\mu_i(\,\cdot\, | x_1,\ldots,x_{i-1})$ the conditional law of $X_i$ knowing $X_1=x_1,X_2=x_2,\ldots,X_{i-1}=x_{i-1}.$ 
The conditional probability measure $\mu_i(\,\cdot\, | x_1,\ldots,x_{i-1})$ is well defined for $\mu_{i-1}$ almost all $(x_1,\ldots,x_{i-1})\in \R^{i-1}$.
When $\mu$ has a positive density $h$ with respect to the Lebesgue measure on $\R^n$, the conditional probability measures $\mu_i(\,\cdot\, | x_1,\ldots,x_{i-1})$ have an explicit density with respect to Lebesgue measure on $\R$ it holds
\[
\int f(u_i)\, \mu_i(du_i | x_1,\ldots,x_{i-1}) = \frac{\int f(u_i) h(x_1,\ldots,x_{i-1},u_i,u_{i+1},\ldots,u_n)\,du_i\cdots d_{u_n}}{\int h(x_1,\ldots,x_{i-1},u_i,u_{i+1},\ldots,u_n)\,du_i\cdots d_{u_n}},
\]
for all bounded continuous $f:\R \to \R.$
\smallskip

The Knothe map $T=(T_1,\ldots,T_n)$ transporting a probability measure $\mu$ on $\R^n$ with a positive density on another probability $\nu$, is defined recursively as follows :
\begin{itemize}
\item[-] $T_1$ is the optimal transport map sending $\mu_1$ on $\nu_1$;
\item[-] for a given $x\in \R^n$, $T_i(x_1,x_2,\ldots,x_{i-1},\,\cdot\,)$ is the one dimensional monotone map sending $\mu_i(\,\cdot\, |x_1,\ldots,x_{i-1})$ on $\nu_i(\,\cdot\,|T_1(x),\ldots,T_{i-1}(x))$.
\end{itemize}
Note that in particular, $T$ is triangular in the sense that $T_i(x)$ depends only on $x_1,\ldots, x_i.$

%Recall the definition of the relative entropy functional (also called Kullback-Leibler divergence): if $\mu,\nu$ are two Borel probability measures on $\R^n$, 
%\[
%D(\nu\,\|\,\mu) = \int \log \frac{d\nu}{d\mu}\,d\nu,
%\]
%if $\nu$ is absolutely continuous with respect to $\mu$ (otherwise, we set $D(\nu\,\|\,\mu)=\infty$). 
%\smallskip

The following lemma is a formally contained in Lemma~\ref{CE}; for completeness, we recall its short proof below. 
\begin{lem}\label{change of variable}
Let $\mu$ be probability measure on $\R^n$ with $\mu(dx)=e^{-V(x)}\,dx$ with $V:\R^n\to\R$ a convex function of class $\mathcal{C}^1$; for all probability measure $\nu$ on $\R^n$ compactly supported with a smooth density, it holds
\begin{align*}
D(\nu\,\|\,\mu) &\geq \int \big[\mathrm{Tr}(DT(x)-I)-\log(|DT(x)|)\big]\,\mu(dx)\\
&= \int \sum_{i=1}^n \big[\partial_i T_i(x)-1 -\log\partial_iT_i(x)\big]\,\mu(dx)
\end{align*}
where $T$ is the Knothe map transporting $\mu$ on $\nu.$
%The same conclusion holds true if $\mu$ is a log-concave probability on $[-1,1]^n$ and $\nu$ has a positive density with respect to $\mu.$
\end{lem}
\proof
Write $g=\frac{d\nu}{dx}$ and $h=\frac{d\nu}{d\mu}$. 
First assume that $T$ is $\mathcal{C}^1$; according to the change of variable formula, it holds
\[
e^{-V(x)}=h(Tx)e^{-V(Tx)}|DT(x)|,
\]
so taking the $\log$ and integrating with respect to $\mu$, we obtain
\[
- \int V(x)\,\mu(dx) = \int \log(h(Tx))\,\mu(dx) - \int V(Tx)\,\mu(dx) + \int \log(|DT(x)|)\,\mu(dx).
\]
So 
\[
D(\nu\,\|\,\mu) = \int V(Tx)-V(x)\,\mu(dx)-\int \log|DT(x)|\,\mu(dx).
\]
By assumption,
\[
V(y)\geq V(x)+\nabla V(x)\cdot (y-x),\qquad \forall x,y\in \R^n.
\]
So,
\[
D(\nu\,\|\,\mu) \geq \int \nabla V(x)\cdot(Tx-x)\,\mu(dx) -\int \log|DT(x)|\,\mu(dx).
\]
Note that, integrating by parts (and using that $\nu$ is compactly supported),
\begin{align*}
\int \nabla V(x)\cdot(Tx-x)\,\mu(dx)= \int \sum_{i=1}^n (\partial_iT(x)-1) e^{-V(x)}\,dx = \int \mathrm{Tr}(DT(x)-I)\,\mu(dx) 
\end{align*}
Thus,
\begin{equation}\label{above with extra term}
D(\nu\,\|\,\mu) \geq \int \big[\mathrm{Tr}(DT(x)-I)-\log|DT(x)|\big]\,\mu(dx).
\end{equation}
Actually the map $T$ is not necessarily of class $\mathcal{C}^1$ so the change of variable formula above needs to be justified. One can consult Section 3 of \cite{Bob07} and invoke for instance \cite[Lemma 3.1]{Bob07}.
%To justify the integration by parts in the cube case, note that
%$$\int_{-1}^1 \partial_i{V}(x)(T_i(x)-x_i)\,dx_i=[V(x)(T_i(x)-x_i)]_{x_i=-1}^{x_i=1} + \int_{-1}^1 V(x)(\partial_i T_i(x)-1)\,dx_i.$$
%For fixed $x_1,\ldots,x_{i-1}$, the map $T_i(x_1,\ldots,x_{i-1},\,\cdot\,)$ is the unique increasing map transporting $\mu_i(\,\cdot\, | x_1,\ldots,x_{i-1})$ on $\nu_i(\,\cdot\,|T_1(x),\ldots,T_{i-1}(x))$. Since these two probabilities have positive densities on $[-1,1]$, it holds $T_1(x_1,\ldots,x_{i-1},-1)=-1$ and $T_i(x_1,\ldots,x_{i-1},1)=1.$ So there are no boundary terms in the integration by parts.
\endproof

\section{A general transport inequality for log-concave probability measures}\label{section transport}

Before introducing our transport cost, we need to briefly discuss on the Cheeger constant (or equivalently, the Poincar\'e constant) of one-dimensional log-concave densities, a case where optimal bounds are known.  If $\gamma$ is a log-concave probability measure on $\R$, denote by $\lambda_\gamma$  its  Cheeger's constant, namely the largest constant for which 
\begin{equation}\label{Cheeger}
 \lambda_\gamma \int |f-m(f)|\,d\gamma \leq \int |f'|\,d\gamma
 \end{equation}
holds for all $f:\R\to \R$ locally-Lipschitz, where $m(f)$ denotes a median of $f$.
It was proven by Bobkov \cite{Bob99} that when $\gamma$ is log-concave probability measure on $\R$, one has 
\begin{equation}\label{Bobkov}
\frac{1}{3\Var(X)}\leq \lambda_\gamma^2 \leq \frac{2}{\Var(X)},
\end{equation}
with $X\sim \gamma$.
Note that if $X$ is a constant random variable, $\Var(X)=0$ and $\lambda=\infty$. 
%The conclusion of the previous estimates is therefore understood with the convention $+\infty \times 0=0.$
\smallskip

In what follows, $\mu$ is a log-concave probability measure on $\R^n$ with full support and $X=(X_1, \ldots, X_n)$ a random vector distributed according to $\mu.$
\smallskip

According to Bobkov's estimate~\eqref{Bobkov}, for all $x\in \R^n$, the one dimensional (log-concave) probability $\mu_{i}(\,\cdot\, | x_1,\ldots,x_{i-1})$ verifies Cheeger's inequality~\eqref{Cheeger} with a constant (optimal up to universal factor) 
\begin{equation}\label{lambda_i}
\lambda_i^2(x)=\lambda_i^2(x_1,\ldots,x_{i-1}) := \frac{1}{3 \Var(X_i | X_1=x_1,\ldots,X_{i-1}=x_{i-1})} \in (0,+\infty]
\end{equation}
where
\begin{multline*}
\Var(X_i | X_1=x_1,\ldots,X_{i-1}=x_{i-1}) \\
= \int u^2 \mu_i(du|x_1,\ldots,x_{i-1}) - \left(\int u\,\mu_i(du|x_1,\ldots,x_{i-1})\right)^2 \in [0,+\infty).
\end{multline*}
In Theorem \ref{transport inequality} below, we prove that any log-concave probability measure on $\R^n$ verifies some transport-entropy inequality with a cost function $c_\mu$ determined by the functions $\lambda_i$ introduced above. In order to state the result, we need to introduce some additional notation.
Recall that if $Z$ is a random vector, we denote by $\overline{Z}$ the random vector defined by 
\[
\overline{Z}_i = Z_i - \E[Z_i | Z_1,\ldots,Z_{i-1}].
\]
Note in particular that $\overline{X}=R(X)$, where the recentering map $R:\R^n\to \R^n$ is defined by $R(x)=(R_1(x),\ldots,R_n(x))$, where
\begin{equation}\label{eq:R}
R_i(x) = x_i - m_i(x),\quad \text{with}\quad m_i(x) = m_i(x_1,x_2,\ldots,x_{i-1}) = \int u\,\mu_i(du | x_1,\ldots,x_{i-1})
\end{equation}
It is not difficult to check that the map $R$ is invertible. We will denote by $S=R^{-1}$ its inverse.
The cost function $c_\mu : \R^n\times \R^n \to [0,\infty]$ is defined as follows,
\[
c_{\mu} (x,y) = \frac{1}{16}\sum_{i=1}^n N\big( \lambda_i(S(x)) (x_i-y_i)\big),\qquad \forall x,y \in \R^n,
\]
where $N(t)=|t|-\log(1+|t|)$ (with the conventions $0\times \infty=0$ and $a\times \infty=\text{sign of a } \times \infty$ for $a\neq0$). The associated optimal transport cost denoted by $\mathcal{T}_\mu$ is defined by
\[
\mathcal{T}_\mu(\nu_1,\nu_2) = \inf_{\pi \in C(\nu_1,\nu_2)} \iint c_\mu(x_1,x_2)\,\pi(dx_1dx_2),
\]
where $C(\nu_1,\nu_2)$ is the set of all probability measures $\pi$ on $\R^n\times \R^n$ such that 
\[
\pi(dx_1\times \R^n)=\nu_1(dx_1)\quad \text{and}\quad \pi(\R^n\times dx_2)=\nu_2(dx_2)
\]

Let us mention that the transport inequality below also holds with the cost  function $\tilde c_\mu : \R^n\times \R^n \to [0,\infty]$ defined as follows
\begin{equation}\label{tildemu}
\tilde c_{\mu} (x,y) = \frac{1}{16}N\left(\sqrt{\sum_{i=1}^n  \lambda_i(S(x))^2(x_i-y_i)^2 }  \,  \right),\qquad \forall x,y \in \R^n.
\end{equation}
Indeed  the function $x\mapsto N(\sqrt{x})$ is subadditive, since it is concave on $\R^+$ and vanishes at $0$, so we have for all $x,y\in \R^n$, 
$c_{\mu} (x,y) \ge \tilde c_{\mu} (x,y) $.

\begin{thm}\label{transport inequality}
Let $X$ be an $n$-dimensional log-concave random vector and let $\mu$ be its law; for all probability measure $\nu$ on $\R^n$ with finite first moment, it holds
\begin{equation}\label{ineq T_mu}
\mathcal{T}_{\mu}(\bar{\mu},\bar{\nu}) \leq  D(\nu\,\|\,\mu),
\end{equation}
where $\bar{\mu}$ is the law of $\overline{X}$ and $\bar{\nu}$ is the law of $\overline{Y}$ with $Y$ distributed according to $\nu.$
%If $\mu$ is supported on $[-1,1]^n$, inequality \eqref{ineq T_mu} holds for all $\nu$ having a positive density with respect to $\mu.$ In this case, all the $\lambda_i$'s are universally bounded and thus we get the inequality
%$$W_2^2(\bar{\mu},\bar{\nu}) \leq a D(\nu\,\|\,\mu),$$
%for all probability measures $\nu$ with a positive density on $[-1,1]^n,$ where $a$ is a universal constant.
\end{thm}

Note that the transport cost depends on $\mu$, and not $\overline \mu$. Indeed, it is given by the  values of $\lambda_i$, which depend on $X\sim \mu$ through formula~\eqref{lambda_i}.

\proof
According to a result by Bobkov and Houdr\'e \cite{BH97}, if $\gamma$ is probability measure on $\R$ verifying Cheeger's inequality \eqref{Cheeger} with constant $\lambda$, then for all convex even function $L:\R\to\R^+$ such that $L(0)=0$, $L(x)>0$ for all $x>0$ and $p_L:=\sup \frac{xL'(x)}{L(x)}<+\infty$, it holds
\[
\int L(f-m(f))\,d\gamma \leq \int L(p_L f'/\lambda)\,d\gamma,
\]
where $m(f)$ denotes the median of $f.$
It will be convenient to replace the median of $f$ by its mean denoted by $\gamma(f)$. First observe that Jensen inequality yields $L(\gamma(f)-m(f))\leq \int L(p_L f'/\lambda)\,d\gamma.$ On the other hand, the convexity of $L$ implies that $$\int L(f-\gamma(f))\,d\gamma \leq \frac{1}{2} \int L(2(f-m(f)))\,d\gamma + \frac{1}{2}L(2(m(f)-\gamma(f))).$$ Finally, it is not difficult to check that the function $L^{1/p_L}$ is subbadditive (see for instance \cite[Lemma 4.7]{GRS13}). It follows that $L(2a)\leq 2^{p_L}L(a)$, $a\geq0.$ Therefore,
\[
\int L(f-\gamma(f))\,d\gamma \leq 2^{p_L} \int L(p_L f'/\lambda)\,d\gamma.
\]
As it is easy to see, for the function $N$, it holds $p_N\leq 2$.  So we have the inequality
\begin{equation}\label{Cheeger gamma}
\frac{1}{16}\int N\left(\lambda(f-\gamma(f))\right) \,d\gamma \leq  \int L(f')\,d\gamma.
\end{equation}
First let us assume that $\mu(dx)=e^{-V(x)}\,dx$ where $V:\R^n\to \R$ is a convex function of class $\mathcal{C}^1$ and $\nu$ is compactly supported with a smooth density.
If $X$ is a random vector of law $\mu$, then using Lemma \ref{change of variable}, the inequality $t-\log(1+t) \geq N(t),$ $t>-1$ and the inequality \eqref{Cheeger gamma}, it holds
\begin{eqnarray*}
D(\nu\,\|\,\mu)&\geq&  \sum_{i=1}^n \E\left[N(\partial_iT_i(X)-1)\right]\\
& =& \sum_{i=1}^n \E\left[\   \E\left[N(\partial_iT_i(X)-1) \,|\, X_1,\ldots,X_{i-1} \right]\   \right]\\
& =& \sum_{i=1}^n \E\left[\   \E\left[N(\partial_i(T_i-x_i)(X)) \,|\, X_1,\ldots,X_{i-1} \right]\   \right]\\
&\geq&  \frac{1}{16} \sum_{i=1}^n \E\Big\{\  \E\big[\ N\left( \lambda_i(X)\left(T_i(X)-\E[T_i(X)|X_1,\ldots,X_{i-1}]- X_i + \E[X_i|X_1,\ldots,X_{i-1}]\right)\right) \\
& &\hskip3cm\big| X_1,\ldots,X_{i-1}\big]    \Big\}\\
&= & \frac{1}{16} \sum_{i=1}^n \E\big[\  N\left( \lambda_i(X)\left(T_i(X)-\E[T_i(X)|X_1,\ldots,X_{i-1}]- X_i + \E[X_i|X_1,\ldots,X_{i-1}]\right)\right)    \big].
\end{eqnarray*}
Note that, since $T_1(X),\ldots,T_{i-1}(X)$ are functions of $X_1,\ldots,X_{i-1}$, it holds
\[
\E[T_i(X) | X_1,\ldots,X_{i-1}] = \E[T_i(X) | T_1(X),\ldots,T_{i-1}(X)]
\]
almost surely.
%by definition of the Knothe map, it holds
%\begin{align*}
%\E[T_i(X)|X_1=x_1,\ldots,X_{i-1}=x_{i-1}] &= \int T_i(x_1,\ldots,x_{i-1},u)\,\mu_i(du|x_1,\ldots,x_{i-1}) \\
%& = \int u\,\nu_i(du|T_1(x),\ldots,T_{i-1}(x)).
%\end{align*}
It follows, that the vector $\overline{Y}$ defined by $\overline{Y}_i = T_i(X)-\E[T_i(X)|X_1,\ldots,X_{i-1}]$ has law $\bar{\nu}.$ 
Using the definition of our cost, we see that
\begin{equation*}\label{eq:a modifier}
D(\nu\,\|\,\mu)\geq  \frac{1}{16}  \E\left[\sum_{i=1}^n  N\left( \lambda_i(S(\overline{X}))\left(\overline{Y}_i- \overline{X}_i\right)\right)    \right]= \E\left[c_\mu(\overline{X},\overline{Y})\right].
\end{equation*}
Therefore, by definition of $\mathcal{T}_\mu$, we have
\[
D(\nu\,\|\,\mu)\geq \mathcal{T}_\mu(\bar{\mu},\bar{\nu}).
\]
Using classical approximation arguments, one extends the inequality above to all probability measures $\nu$ with finite finite first moment.
\smallskip

This completes the proof of Theorem \ref{transport inequality} in the case where $\mu(dx)=e^{-V(x)}\,dx$ with a convex function $V$ of class $\mathcal{C}^1$ on $\R^n.$
The conclusion is then extended, using classical approximation arguments, to any $\mu(dx)=e^{-V(x)}\,dx$ where $V:\R^n \to \R\cup \{+\infty\}$ is a lower semi-continuous convex function whose domain has a non empty interior.
A way to do it is to consider the family of convex functions $V_s$, $s>0$ defined by
\[
V_s(x) = \inf_{y\in \R^n}\left\{V(y) + \frac{1}{s} |x-y|^2\right\},\qquad x\in \R^n,\ s>0.
\]
It is well known that for all $s>0$, $V_s:\R^n \to \R$ is a $\mathcal{C}^1$ smooth convex function on $\R^n$ converging monotonically to $V$ as $s\to 0$ (see for instance \cite[Theorem 4.1.4]{HUL93}). Details are left to the reader.
\endproof

\proof[Proof of Theorem \ref{thm:transport T2}] 
Assume that $\mu$ is the law of an $n$-dimensional log-concave random vector $X$ taking values in the hypercube $Q=[-R,R]^n$. This assumption on the support of $\mu$ has for consequence that for all $x\in Q$,
\[
\Var(X_i | X_1=x_1,\ldots,X_{i-1}=x_i) \leq 2R^2.
\]
Therefore, $\lambda_i(x) \geq 1/(\sqrt{6}R)$ for all $i\in \{1,\ldots,n\}$ and $x\in Q$. It is not difficult to check that there is an absolute constant $c>0$ such that $N(u)\geq cu^2$ for all $|u| \leq 2/\sqrt{6}$. So, if $\nu$ is a given probability measure on $Q$, then by Theorem~\ref{transport inequality} it holds
\[
D(\nu\,\|\,\mu)\geq \frac{c}{R^2} \E\left[|\overline{X}-\overline{Y}|^2\right],
\]
which completes the proof.
\endproof

\section{Weighted Poincar\'e inequalities for log-concave probability measures}
In this last section, we use a classical linearization technique  to prove that the transport cost inequality obtained in Theorem \ref{transport inequality} implies the weighted Poincar\'e inequality of Theorem \ref{main result}. Such linearization depends only on the behavior of the cost for small distances. It will be more convenient, notationnaly speaking, but equivalent, to use the cost $\tilde c_\mu$ defined by~\eqref{tildemu}  in the definition of $\mathcal T_\mu$ and in Theorem~\ref{transport inequality}, rather than $c_\mu$.
\smallskip

Let us introduce the following supremum convolution operator 
\[
P_tf(x) = \sup_{y\in \R^n}\left\{f(y) - \frac{1}{t} \tilde c_\mu(x,y)\right\},\qquad \forall x\in \R^n,\qquad \forall t>0,
\]
which is well defined for instance for bounded continuous function $f$ on $\R^n$. It can be shown that he function $u(t,x) = P_tf(x)$ satisfies in some weak sense the following Hamilton-Jacobi equation
\[
\partial_t u(t,x) =  8 \sum_{i=1}^n \frac{1}{\lambda_i^2(S(x))} \left(\partial_{x_i} u\right)^2(t,x).
\]
Actually, in what follows, we will only need the following elementary inequality:
\begin{lem}\label{HJ}
For all differentiable bounded Lipschitz function $f:\R^n \to \R$, 
\[
\limsup_{t \to 0} \frac{1}{t} \int P_tf-  f \,d\nu \leq 8  \int \sum_{i=1}^n \frac{1}{\lambda_i^2(S(x))} \left(\partial_{x_i} f\right)^2(x)\,\nu(dx),
\]
for all probability measure $\nu$ on $\R^n$ such that $\int \lambda_i^{-2}(S)\,d\nu$ is finite for every $i\in \{1,\ldots,n\}$.
\end{lem}
Let us admit the lemma for a moment and prove Theorem \ref{main result}.
\proof[Proof of Theorem \ref{main result}]
%First let us assume that $X$ has law $\mu(dx) = e^{-V(x)}\,dx$, where $V$ is a convex function of class $\mathcal{C}^1.$ 
Let $g:\R^n \to \R$ be a bounded function such that $\int g\,d\mu = 0$ and define for all $t\geq 0$ the measure $\nu^{\, t}(dx)= (1+tg)\,\mu(dx)$. If $t$ is small enough, then $\nu^{\, t}$ is a probability measure.
Let $\pi$ be a coupling of $\bar{\mu}$ and $\overline{\nu^{\, t}}$, and $a>0$ be a parameter whose value will be fixed later on ; for all bounded differentiable Lipschitz function $f:\R^n \to \R$, it holds
\begin{align*}
\frac{1}{t}\left(\int f d\overline{\nu^{\, t}} - \int f\,d\overline{\mu} \right)&= \frac{1}{t} \int f(y)-f(x)\,\pi(dxdy)\\
& =  \frac{1}{t} \int f(y)-P_{at}f(x)\,\pi(dxdy) + \frac{1}{t} \int P_{at}f(x) -f(x)\,\overline{\mu}(dx)\\
& \leq  \frac{1}{at^2} \int \tilde c_\mu(x,y)\,\pi(dxdy) + \frac{1}{t} \int P_{at}f(x) -f(x)\,\overline{\mu}(dx),
\end{align*}
where the last line comes from the inequality $f(y)-P_sf(x)  \leq \frac{1}{s}\tilde c_{\mu}(x,y)$ for all $s>0.$
So optimizing over $\pi \in C(\bar{\mu},\overline{\nu^{\, t}})$, we get, for all $t$ small enough,
\begin{align*}
\frac{1}{t}\left(\int f d\overline{\nu^{\, t}} - \int f\,d\overline{\mu} \right)& \leq \frac{1}{at^2} \mathcal{T}_{\mu}(\bar{\mu},\overline{\nu^{\, t}}) +  \frac{1}{t} \int P_{at}f(x) -f(x)\,\overline{\mu}(dx)\\
& \leq \frac{1}{at^2} D(\nu^{\, t}\, \|\, \mu) +  \frac{1}{t} \int P_{at}f(x) -f(x)\,\overline{\mu}(dx),
\end{align*}
where the last inequality comes from Theorem \ref{transport inequality}. A straightforward calculation shows that $t^{-2}D(\nu^{\, t} \,\|\,\mu) \to \frac{1}{2} \int g^2\,d\mu$ when $t$ goes to $0.$ Therefore, using Lemma \ref{HJ}, we get
\[
\limsup_{t\to 0} \frac{1}{t}\left(\int f d\overline{\nu^{\, t}} - \int f\,d\overline{\mu} \right) \leq \frac{1}{2a} \int g^2\,d\mu + 8a  \int \sum_{i=1}^n \frac{1}{\lambda_i^2(S(x))} \left(\partial_{x_i} f\right)^2(x)\,\bar{\mu}(dx)
\]
and optimizing over $a>0$
\[
\limsup_{t\to 0} \frac{1}{t}\left(\int f d\overline{\nu^{\, t}} - \int f\,d\overline{\mu} \right) \leq 4  \left(\int g^2\,d\mu\right)^{1/2} \left(\int \sum_{i=1}^n \frac{1}{\lambda_i^2(S(x))} \left(\partial_{x_i} f\right)^2(x)\,\bar{\mu}(dx)\right)^{1/2}
\]
Now let us evaluate the left hand side. 
%If $Y^t$ is a random vector distributed according to $\nu^{\, t}$, we denote by $\nu_{t,1}$ the  law of $Y^t_1$ and for all $i \in \{2,\ldots,n\}$, let $\nu_{t,i}(\,\cdot\, | y_1,\ldots,y_{i-1})$ be the conditional distribution of $Y^y_i$ knowing $Y^t_1=y_1,\ldots, Y^t_{i-1}=y_{i-1}$ . 
Consider the map $R^t$ defined by 
\[
R^t(x) =\left[x_1 - \int u_1\,\nu^{\, t}_{1}(du), x_2 - \int u_2 \,\nu^{\, t}_{2}(du_2 | x_1),\ldots,x_n - \int u_n\,\nu^{\, t}_{n}(du_n|x_1,\ldots,x_{n-1}) \right].
\]
For $t=0$, $R^0=R$ is the map introduced in \eqref{eq:R}.
Then $\overline{\nu^{\, t}}$ is the image of $\nu^{\, t}$ by the map $R^t$ and $\bar{\mu}$ the image of $\mu$ by the map $R.$ Therefore,
\begin{align}
\frac{1}{t}\left(\int f d\overline{\nu^{\, t}} - \int f\,d\overline{\mu} \right) &= \frac{1}{t}\left(\int f(R^t) (1+tg)\,d\mu  - \int f(R)\,d\mu \right)\notag\\
& \to -\int \nabla f(R) \cdot U\,d\mu + \int f(R)g\,d\mu,\label{limite a justifier}
\end{align}
when $t$ goes to $0$, where $U:= \lim_{t \to 0} \frac{1}{t}(R^t-R).$ Let us calculate $U$.
Writing the definition, it is not difficult to see that,
\[
\int u_i\,\nu_i^{\,t}(du_i | x_1,\ldots, x_{i-1}) = \frac{a_i+t b_i}{c_i+ td_i},
\]
with 
\begin{align*}
a_i&= \int u_i e^{-V(x_1,\ldots,x_{i-1}, u_i,\ldots,u_n) }\,du_i\cdots du_n,\\
b_i& = \int u_i g(x_1,\ldots,x_{i-1},u_i,\ldots,u_n)e^{-V(x_1,\ldots,x_{i-1}, u_i,\ldots,u_n) }\,du_i\cdots du_n,\\
c_i &=\int  e^{-V(x_1,\ldots,x_{i-1}, u_i,\ldots,u_n) }\,du_i\cdots du_n,\\
d_i&=\int g(x_1,\ldots,x_{i-1},u_i,\ldots,u_n)e^{-V(x_1,\ldots,x_{i-1}, u_i,\ldots,u_n) }\,du_i\cdots du_n.
\end{align*}
Therefore,
\begin{align*}
U_i(x) &= \lim_{t\to 0}\frac{1}{t}\left( \int u_i\,d\nu^{\,t}_i(du_i | x_1,\ldots,x_{i-1}) -  \int u_i\,d\mu_i(du_i| x_1,\ldots,x_{i-1}) \right)\\
& = \lim_{t \to 0} \frac{1}{t} \left(  \frac{a_i+t b_i}{c_i+ td_i}-\frac{a_i}{c_i} \right) = \frac{b_i}{c_i}-\frac{a_i}{c_i} \frac{d_i}{c_i} \\
&= \E[X_ig(X)| X_1=x_1,\ldots,X_{i-1}=x_{i-1}]\\
&\quad- \E[X_i| X_1=x_1,\ldots,X_{i-1}=x_{i-1}]\cdot \E[g(X)|X_1=x_1,\ldots,X_{i-1}=x_{i-1}].
\end{align*}
It is easy to check that $\frac{1}{t}\left| \frac{a_i+tb_i}{c_i+td_i} - \frac{a_i}{c_i} \right| \leq \frac{2M}{1-tM}\frac{|a_i|}{c_i}$ for $t$ sufficiently small, where $M = \sup|g|$. This can be used to justify the limit in \eqref{limite a justifier}. Details are left to the reader.

According to what precedes, 
\[
U(X)=\E_{i-1}[X_ig(X)] -\E_{i-1}[X_i]\E_{i-1}[g(X)]  = \E_{i-1}[\overline{X}_ig(X)].
\]
So putting everything together, we get 
\begin{align*}
\E[f(\overline{X})g(X)] &\leq 4\E[g^2(X)]^{1/2}\E\left[\sum_{i=1}^n \frac{1}{\lambda_i^2(S(\overline{X}))} \left(\partial_i f(\overline{X})\right)^2\right]^{1/2}
+ \sum_{i=1}^n \E\left[ \E_{i-1}[\overline{X}_ig(X)] \partial_i f(\overline{X}) \right] \\
& = 4\sqrt{3}\E[g^2(X)]^{1/2}\E\left[\sum_{i=1}^n \E\left[\overline{X}_i^2|\overline{X}_1,\ldots,\overline{X}_{i-1}\right] \left(\partial_i f(\overline{X})\right)^2\right]^{1/2}
+ \sum_{i=1}^n \E\left[ \E_{i-1}\left[\overline{X}_ig(X)\right] \partial_i f(\overline{X}) \right],
\end{align*}
where the second line comes from the definition of the $\lambda_i$'s and the identity
\begin{align*}
\Var(X_i | X_1=S_1(\bar{x}),\ldots,X_{i-1}=S_{i-1}(\bar{x}))  &= \E[\overline{X}_i^2 | X_1=S_1(\bar{x}),\ldots,X_{i-1}=S_{i-1}(\bar{x})]\\
&=  \E\left[\overline{X}_i^2 | \overline{X}_1=\bar{x}_1,\ldots,\overline{X}_{i-1}=\overline{x}_{i-1}\right],
\end{align*}
for all $\bar{x}=(\bar{x}_1,\ldots,\overline{x}_{n}) \in \R^n$.
Finally let us bound the last term. Using Cauchy-Schwarz, it holds
\begin{align*}
\sum_{i=1}^n \E\left[ \E_{i-1}[\overline{X}_ig(X)] \partial_i f(\overline{X}) \right]&= \sum_{i=1}^n \E\left[ \overline{X}_ig(X) \E_{i-1}[\partial_i f(\overline{X})] \right]\\
& \leq \E[g^2(X)]^{1/2}\E\left[ \left(\sum_{i=1}^n \overline{X}_i\E_{i-1}[\partial_if (\overline{X})]    \right)^2\right]^{1/2}.
\end{align*}
Now observe that if $i\leq j-1$, then, since $\overline{X}_i\E_{i-1}[\partial_if (\overline{X})]\E_{j-1}[\partial_if (\overline{X})]$ is a function of $X_1,\ldots,X_{j-1}$, it holds
\[
\E\left[  \overline{X}_i\E_{i-1}[\partial_if (\overline{X})] \cdot \overline{X}_j\E_{j-1}[\partial_if (\overline{X})\right] = \E\left[  \overline{X}_i\E_{i-1}[\partial_if (\overline{X})] \E_{j-1}[\partial_if (\overline{X})] \cdot \E_{j-1}[\overline{X}_j]\right]=0,
\]
since $\E_{j-1}[\overline{X}_j]=0.$ Therefore, 
\begin{align*}
\sum_{i=1}^n \E\left[ \E_{i-1}[\overline{X}_ig(X)] \partial_i f(\overline{X}) \right]& \leq \E[g^2(X)]^{1/2}\E\left[ \sum_{i=1}^n \overline{X}_i^2\E_{i-1}[\partial_if (\overline{X})]^2 \right]^{1/2}\\
& \leq  \E[g^2(X)]^{1/2}\E\left[ \sum_{i=1}^n \overline{X}_i^2\E_{i-1}[\partial_if (\overline{X})^2] \right]^{1/2}\\
%& = \E[g^2(X)]^{1/2}\E\left[ \sum_{i=1}^n \overline{X}_i^2\E_{i-1}[\partial_if (\overline{X})]^2 \right]^{1/2}\\
& = \E[g^2(X)]^{1/2}\E\left[ \sum_{i=1}^n \E_{i-1}\left[\overline{X}_i^2\right] \partial_if (\overline{X})^2 \right]^{1/2}.
\end{align*}
Using again that $\E_{i-1}\left[\overline{X}_i^2\right] = \E\left[\overline{X}_i^2 | \overline{X}_1,\ldots,\overline{X}_{i-1}\right]$, we get
\[
\E[f(\overline{X})g(X)] \leq \left(4\sqrt{3}+1\right)\E[g^2(X)]^{1/2}\E\left[\sum_{i=1}^n \E\left[\overline{X}_i^2 | \overline{X}_1,\ldots,\overline{X}_{i-1}\right] \left(\partial_i f(\overline{X})\right)^2\right]^{1/2}.
\]
Taking $g=f\circ R$ with $f$ such that $\int f\,d\bar{\mu}=0$, we obtain
\[
\E[f(\overline{X})^2] \leq \left(4\sqrt{3}+1\right)^2\E\left[\sum_{i=1}^n \E\left[\overline{X}_i^2 | \overline{X}_1,\ldots,\overline{X}_{i-1}\right] \left(\partial_i f(\overline{X})\right)^2\right].
\]
The inequality is then extended by density to all locally Lipschitz functions $f:\R^n\to \R$ such that $\int f^2\,d\mu<\infty.$
\endproof

It remains to prove Lemma \ref{HJ}.
\proof[Proof of Lemma \ref{HJ}]
Let $f:\R^n \to \R$ be a differentiable bounded Lipschitz function and denote by $M=1+\sup |f|$. 
For all $x\in \R^n$, we denote by $\|\,\cdot\,\|_x$ the quantity defined by
\[
\|u\|_x = \sqrt{\sum_{i=1}^n \lambda_i^2(S(x))u_i^2},\qquad \forall u\in \R^n.
\]
When $x$ is such that $\lambda_i(x)<\infty$ for all $i$, then $\|\,\cdot\,\|_x$ is a norm on $\R^n.$
%We also 
With this notation $\tilde c_\mu(x,y) = \frac{1}{16} N(\|x-y\|_x).$
%Note that the function $N$ is such that $x\mapsto N(\sqrt{x})$ is concave on $[0,\infty)$. So it holds
%\[
%\sum_{i=1}^n N\left( \frac{\lambda_i(x)}{2} (x_i-y_i)\right) \geq N\left( \frac{\|x-y\|_x}{2}\right),\qquad \forall y\in \R^n.
%\]
and 
\[
P_tf(x) = \sup_{y\in \R^n}\left\{ f(y) - \frac{1}{16t} N\left(\|y-x\|_x\right)   \right\}.
\]

First, note that, for all $x \in \R^n$, the supremum in the definition of $P_t f(x)$ is attained on the ball $\{y \in \R^n ; \|y-x\|_x \leq N^{-1}(48Mt)\}.$ Namely, if $y$ is outside the ball, it holds
\begin{align*}
f(y)-f(x)-\frac{1}{16t}N\left( \|y-x\|_x\right) &\leq -M<0.
\end{align*}
Since $P_tf(x)\geq f(x)$, we conclude that the supremum is reached inside the ball.

Now let us bound from above the derivative of $P_tf$ with resp  ect to the $t$ variable.
Let $x\in \R^n$ be such that $\lambda_i(x)<\infty$ for all $i$. Using the preceding remark and the inequality $uv\leq \frac{u^2}{2}+\frac{v^2}{2},$ we see that
\begin{align}
\frac{1}{t}(P_tf(x) - f(x))& = \sup_{\| y-x\|_x \leq N^{-1}(48Mt)}\left\{ \frac{f(y)-f(x)}{t} -\frac{1}{16t^2}N(\|y-x\|_x) \right\}\notag\\
& \leq 4\sup_{\| y-x\|_x \leq N^{-1}(48Mt)}\left\{ \frac{(f(y)-f(x))^2}{N(\|y-x\|_x)} \right\}\label{eq:lemme HJ}\\
& =4\sup_{ \| u\|_x \leq N^{-1}(48Mt)}\left\{ \frac{(\nabla f(x)\cdot u)^2 + o(\|u\|_x^2)}{ \frac{\|u\|_x^2}{2} +o(\|u\|_x^2)} \right\}\notag\\
&\leq  8\sup_{ \| u\|_x \leq N^{-1}(48Mt)}\left\{ \frac{\sum_{i=1}^n \frac{1}{\lambda_i^2(S(x))}(\partial_if)(x)^2 + o(1)}{ 1 +o(1)} \right\}\to 8\sum_{i=1}^n \frac{1}{\lambda_i^2(S(x))}(\partial_if)(x)^2,\notag
\end{align}
when $t$ goes to $0$, where the last inequality follows from Cauchy-Schwarz. So we conclude that if $x$ is such that $\lambda_i(x)<\infty$ for all $i$, then
\[
\limsup_{t\to 0} \frac{1}{t}(P_tf(x)-f(x)) \leq 8\sum_{i=1}^n \frac{1}{\lambda_i^2(S(x))}(\partial_if)(x)^2.
\]
If $x$ is such that $\lambda_i(x)=0$ for some $x$, then $P_tf(x) = f(x)$ and so the inequality above is still true.

Moreover, denoting by $\lambda^*(x) = \min \{\lambda_i(S(x))\}>0$, it follows from \eqref{eq:lemme HJ} and from the inequality $\|u\|_x \geq \lambda^*(x)|u|$, $u\in \R^n$, that for all $t\leq 1/(48M)$
\begin{equation}\label{Fatou}
\frac{1}{t}(P_tf(x) - f(x)) \leq  4 \sup_{\lambda^*(x)|y-x|\leq N^{-1}(1)} \left\{ \frac{(f(y)-f(x))^2}{N(\lambda^*(x) |y-x|)}\right\}\leq a\frac{L^2}{\lambda^*(x)^2},
\end{equation}
where $L$ the Lipschitz constant of $f$ and $a= 4\sup_{0<v\leq N^{-1}(1) } \frac{v^2}{N(v)}$ is some universal constant. 

Now, let $\nu$ be a probability measure on $\R^n$ such that $\int \frac{1}{\lambda_i^2(S(x))}\,\nu(dx)<+\infty$ for all $i$. Then $1/\lambda^*$ is also square integrable with respect to $\nu$. Therefore, thanks to \eqref{Fatou}, one can apply Fatou's lemma in its $\limsup$ form:
\[
\limsup_{t \to 0} \int \frac{1}{t}(P_tf -f)\,d\nu \leq \int \limsup_{t\to 0} \frac{1}{t}(P_tf - f)\,d\nu \leq 8\int \sum_{i=1}^n \frac{1}{\lambda_i^2(S(x))}(\partial_if)(x)^2\,d\nu.
\]
\endproof

\section{Variance estimates}
Here we prove Corollary \ref{cor:variance}, identity~\eqref{identity} and Corollary \ref{encadrement}.
\begin{proof}[Proof of Corollary \ref{cor:variance}]
%Under the condition $\E_{i-1}[X_i]=0$ for all $i\in \{1,\ldots,n\}$, the function $f(x)=|x|^2$ satisfies
%\[
%\E[\E_{i-1}[X_if(X)] \partial_if(X)]=2\E[\E_{i-1}[X_if(X)] X_i] = 2\E[\E_{i-1}[X_if(X)] \E_{i-1}[X_i]] = 0.
%\]
%Therefore, 
According to Theorem \ref{main result} and standard properties of conditional expectations, it holds 
\begin{align*}
\Var(|\overline{X}|^2) &\leq 4a \sum_{i=1}^n  \E\left[ \E[\overline{X}_i^2|\overline{X_1},\ldots,\overline{X}_{i-1}] \overline{X}_i^2\right]
= 4a \sum_{i=1}^n  \E\left[ \E[\overline{X}_i^2|\overline{X_1},\ldots,\overline{X}_{i-1}]^2 \right]\\
&\leq 4a \sum_{i=1}^n  \E\left[ \E[\overline{X}_i^4|\overline{X_1},\ldots,\overline{X}_{i-1}] \right]
 = 4a \sum_{i=1}^n  \E\left[ \overline{X_i}^4 \right].
%\leq 4aa' \sum_{i=1}^n  \E\left[ X_i^2 \right]=4aa'n,
\end{align*}
Observe that $\E[\overline{X}_i^4] \leq 8 \E[X_i^4] + 8\E\left[ \E_{i-1}[X_i]^4\right] \leq 16 \E[X_i^4].$
We conclude using Borell's reverse H\"older inequality \cite{Bor74}: there exists some universal constant $a'$ such that for any log-concave random variable $Y$, it holds $\E[Y^4]\leq a' \E[Y^2]^2$.
So,
\[
\Var(|\overline{X}|^2) \leq 64aa'n.
\]
\end{proof}

\begin{rem}\label{rem BCE}
Our main result Theorem \ref{main result} is closely related to a result by Barthe and Cordero-Erausquin \cite{BCE13}. Namely, it follows from \cite[Theorem 4]{BCE13} that if $X$ is a random vector following a law $\mu(dx)=e^{-V(x)}\,dx$ on $\R^n$ with $\mathrm{Hess}\,V \geq \rho \mathrm{Id}$ for some $\rho \geq0$, then for all smooth function $f:\R^n \to \R$ such that
\begin{equation}\label{eq:BCE}
\E[\partial_if(X) | X_1\ldots,X_{i-1},X_{i+1},\ldots,X_n]=0,\qquad \forall i \in \{1,\ldots, n\}
\end{equation}
it holds
\[
\Var_\mu(f(X)) \leq \sum_{i=1}^n\E\left[  \frac{1}{\rho + 1/C_i(X)} \partial_i f^2(X)\right],
\]
where, for all $x=(x_1,\ldots,x_n)\in \R^n$, $C_i(x)$ denotes the Poincar\'e constant of the conditional distribution of $X_i$ knowing $X_1=x_1,\ldots, X_{i-1}=x_{i-1},X_{i+1}=x_{i+1},\ldots,X_n=x_n$. Note that the conclusion of \cite[Theorem 4]{BCE13} is more general than what we state above. In the general formulation, to any decomposition of the identity $\mathrm{Id} = \sum_{i=1}^m c_i P_{E_i}$ where $c_i>0$ and $P_{E_i}$ is the orthogonal projection on a subspace $E_i$ corresponds a weighted Poincar\'e inequality involving the Poincar\'e constants of the conditional distributions of $X$ knowing $P_{F_i}(X)$, with $F_i=E_i^{\bot}$.
\smallskip

It is well known (see for instance Theorem \ref{Bobkov} below) that Poincar\'e constants of one dimensional log-concave probability measures can be estimated by the variance. In particular, it holds
\begin{align*}
C_i(x) \leq 3 \Var(X_i |X_1=x_1,\ldots, X_{i-1}=x_{i-1},X_{i+1}=x_{i+1},\ldots,X_n=x_n),\qquad \forall i\in \{1,\ldots,n\}.
\end{align*}
Therefore, taking $\rho=0$, it holds
\begin{equation}\label{eq:BCE2}
\Var_\mu(f(X)) \leq 3\sum_{i=1}^n\E\left[ \Var(X_i| X_1\ldots,X_{i-1},X_{i+1},\ldots,X_n] \partial_i f^2(X)\right],
\end{equation}
for all smooth $f$ enjoying \ref{eq:BCE}. The difference between this result and Theorem \ref{main result} (besides the fact that we condition only with respect to the first variables) is that our result is true for all $f$ but for $\overline{X}$ instead of $X.$
\smallskip

Let us see how to recover the conclusion of Corollary \ref{cor:variance} from \eqref{eq:BCE2}. 
Let us assume that $X$ is such that 
\[
\E[X_i | X_1\ldots,X_{i-1},X_{i+1},\ldots,X_n]=0,\qquad \forall i\in \{1,\ldots,n\}.
\]
This condition (which is actually a bit stronger than the condition $\overline{X}=X$) exactly amounts to require that the function $f(x)=|x|^2$ satisfies \eqref{eq:BCE}. So applying \eqref{eq:BCE2} to this function and reasoning as in the proof of Corollary \ref{cor:variance} we thus get from \eqref{eq:BCE2} that $\Var(|X|^2) \leq a n$ for some universal constant $a.$ 
\end{rem}

\begin{proof}[Proof of Corollary \ref{encadrement}] Let us start with identity~\eqref{identity}.
For all $i\in \{1,\ldots,n\}$, it holds
\[
\E\left[\overline{X}_i\,\E_{i-1}[X_i]\right] = \E\Big[ (X_i-\E_{i-1}[X_i])\,\E_{i-1}[X_i] \Big]=\E\Big[\E_{i-1}[X_i-\E_{i-1}[X_i]]\ \E_{i-1}[X_i]\Big]=0.
\]
As a result $\overline{X}$ and $X'$ are orthogonal in $\mathbb{L}_2(\Omega,\mathcal{A},\P ; \R^n)$. 
Therefore, it holds 
\[
\E[|X|^2] = \E[|\overline{X}|^2] + \E[|X'|^2] \quad\text{and}\quad \E[|X|^2]^2 = \E[|\overline{X}|^2]^2 +2\E[|\overline{X}|^2]\E[|X'|^2] + \E[|X'|^2]^2.
\] 
Since
\begin{align*}
\E[|X|^4] &= \E[|\overline{X}|^4] + 2\E[|\overline{X}|^2]\E[|X'|^2] + \E[|X'|^4] + 4 \E[(\overline{X}\cdot X')^2] + 4\E[|\overline{X}|^2\overline{X}\cdot X'] + 4\E[|X'|^2\overline{X}\cdot X'],
\end{align*}
we get that
\[
\Var(|X|^2) = \Var(|\overline{X}|^2) + \Var(|X'|^2) + 2 \mathrm{Cov}(|\overline{X}|^2,|X'|^2) + 4 \E[(\overline{X}\cdot X')^2] + 4\E[|\overline{X}|^2\overline{X}\cdot X'] + 4\E[|X'|^2\overline{X}\cdot X']
\]
Using Cauchy-Schwarz, and the orthogonality of $\overline{X}$ and $X'$ we get
\begin{align*}
\left|\mathrm{Cov}\left(|\overline{X}|^2,|X'|^2\right)\right| &\leq \sqrt{\Var\left(|\overline{X}|^2\right)}\sqrt{\Var\left(|X'|^2\right)}\\
\left| \E\left[|\overline{X}|^2\,\overline{X}\cdot X'\right]\right|& = \left|\E\left[\left(|\overline{X}|^2-\E[|\overline{X}|^2\right)\, \overline{X}\cdot X'\right]\right| \leq \sqrt{\Var(|\overline{X}|^2)}\sqrt{\E[(\overline{X}\cdot X')^2]}\\
\left|\E\left[|X'|^2\,\overline{X}\cdot X'\right]\right|&=\left|\E\left[\left(|X'|^2-\E[|X'|^2\right)\, \overline{X}\cdot X'\right]\right| \leq \sqrt{\Var(|X'|^2)}\sqrt{\E[(\overline{X}\cdot X')^2]}.
\end{align*}
Moreover, note that if $i<j$ the random variable $ \overline{X}_i\E_{i-1}[X_i]\E_{j-1}[X_j]$ is measurable with respect to the $\sigma$ field generated by $X_1,\ldots,X_{j-1}$. Therefore 
\[
\E\left[ \overline{X}_i\E_{i-1}[X_i]\overline{X}_j\E_{j-1}[X_j] \right] = \E\left[ \overline{X}_i\E_{i-1}[X_i]\E_{j-1}[\overline{X}_j]\E_{j-1}[X_j] \right] = 0
\]
So, it holds
\begin{align*}
\E[(\overline{X}\cdot X')^2]& = \sum_{i=1}^n \E\left[\overline{X}_i^2\E_{i-1}[X_i]^2\right] + 2 \sum_{i<j}\E\left[ \overline{X}_i\E_{i-1}[X_i]\overline{X}_j\E_{j-1}[X_j] \right]\\
& = \sum_{i=1}^n \E\left[\overline{X}_i^2\E_{i-1}[X_i]^2\right] = \sum_{i=1}^n \E\left[ \E_{i-1}[X_i^2] \E_{i-1}[X_i]^2 - \E_{i-1}[X_i]^4\right]\\
& \leq \sum_{i=1}^n \E[X_i^4] \leq a'n,
\end{align*}
where $a'$ is some universal constant such that $\E[Y^4]\leq a' \E[Y^2]^2$ for all log-concave random variable $Y.$
We conclude from the inequalities above that
\begin{align*}
\Var(|X|^2) &\leq \left( \sqrt{\Var(|\overline{X}|^2)} + \sqrt{\Var(|X'|^2)}\right)^2 + 4a'n +4\sqrt{a'n}\left( \sqrt{\Var(|\overline{X}|^2)} + \sqrt{\Var(|X'|^2)}\right)\\
&=\left( \sqrt{\Var(|\overline{X}|^2)} + \sqrt{\Var(|X'|^2)}+2\sqrt{a'n}\right)^2\leq \left( \sqrt{\Var(|X'|^2)}+(2\sqrt{a'}+b)\sqrt{n}\right)^2\\
& \leq 2 \Var(|X'|^2) + 2(2\sqrt{a'}+b)^2n
\end{align*}
where in the last inequalities $b$ is the universal constant given by Corollary \ref{cor:variance}.

Similarly,
\begin{align*}
\Var(|X|^2) &\geq \left(\sqrt{\Var(|X'|^2)}-\sqrt{\Var(|\overline{X}|^2)} \right)^2 -4\sqrt{a'n}\left(\sqrt{\Var(|X'|^2)}+\sqrt{\Var(|\overline{X}|^2)} \right).
\end{align*}
Therefore, expanding the square, we see that the number $\sqrt{V'}:=\sqrt{\Var(|X'|^2)}$ is less than or equal the positive root of the equation
\[
x^2 - 2x\left(\sqrt{\bar{V}}+2\sqrt{a'n}\right) + \bar{V}-4\sqrt{a'n}\sqrt{\bar{V}}-V=0,
\]
with $V=\Var(|X|^2)$ and $\bar{V}=\Var(|\overline{X}|^2).$ An easy calculation thus gives
\[
\sqrt{V'} \leq \sqrt{\bar{V}} +2\sqrt{a'n} + \sqrt{4a'n +V},
\]
which together with Corollary \ref{cor:variance} easily gives the desired inequality.
\end{proof}

%\end{align*}
\bibliographystyle{amsplain}

\end{document}